\documentclass[11pt]{amsart}

\usepackage[a4paper,hmargin=3.5cm,vmargin=4cm]{geometry}
\usepackage{amsfonts,amssymb,amscd,amstext}
\usepackage{graphicx}
\usepackage[dvips]{epsfig}
\usepackage{quoting}
\usepackage{hyperref}  

\usepackage{fancyhdr}
\usepackage{enumerate}
\usepackage{titlesec}
\usepackage{mathrsfs}

\titleformat{\section}
{\filcenter\bfseries\large} {\thesection{.}}{0.2cm}{}
\titleformat{\subsection}[runin]
{\bfseries} {\thesubsection{.}}{0.15cm}{}[.]
\titleformat{\subsubsection}[runin]
{\em}{\thesubsubsection{.}}{0.15cm}{}[.]

\usepackage[up,bf]{caption}

\newtheorem{theorem}{Theorem}[section]

\newtheorem{lemma}[theorem]{Lemma}
\newtheorem{claim}[theorem]{Claim}
\newtheorem{corollary}[theorem]{Corollary}

\theoremstyle{definition}
\newtheorem{definition}[theorem]{Definition}

\numberwithin{equation}{section}
\numberwithin{figure}{section}

\newcommand\Bcal{\mathcal{B}}
\newcommand\Ccal{\mathcal{C}}

\newcommand\Ocal{\mathcal{O}}
\newcommand\Pcal{\mathcal{P}}

\newcommand\C{\mathbb{C}}

\newcommand\R{\mathbb{R}}

\renewcommand\c{\mathbb{C}}

\newcommand\cp{\mathbb{CP}}

\newcommand\n{\mathbb{N}}
\renewcommand\r{\mathbb{R}}
\newcommand\s{\mathbb{S}}

\newcommand\z{\mathbb{Z}}

\newcommand\igot{\mathfrak{i}}

\renewcommand\igot{\mathfrak{i}}

\newcommand\Agot{\mathfrak{A}}
\newcommand\Igot{\mathfrak{I}}

\newcommand\Fgot{\mathfrak{F}}
\newcommand\Sgot{\mathfrak{S}}

\renewcommand\imath{\igot}

\newcommand\di{\partial}

\renewcommand\span{\mathrm{span}}

\newcommand\Div{\mathrm{Div}}

\newcommand\Flux{\mathrm{Flux}}

\newcommand\supp{\mathrm{supp}}
\newcommand\Qcal{\mathcal{Q}}

\def\span{\mathrm{span}}

\def\Flux{\mathrm{Flux}}

\usepackage{color}

\begin{document}


\fancyhead[LO]{Complete non-orientable minimal surfaces of finite total curvature}
\fancyhead[RE]{A.\ Alarc\'on and F.J. L\'opez}
\fancyhead[RO,LE]{\thepage}

\thispagestyle{empty}



\begin{center}
{\bf\Large Complete non-orientable minimal surfaces of finite total curvature: Approximation and interpolation}

\vspace*{5mm}

%
%
{\large\bf Antonio Alarc\'on\; and\; Francisco J. L\'opez}

\vspace*{2.5mm}
{\em In memory of William H. Meeks III}
\end{center}

%
%
\vspace*{5mm}

{\small
\noindent {\bf Abstract}\hspace*{0.1cm} 
In this paper, we prove a uniform approximation theorem with interpolation for complete non-orientable minimal surfaces of finite total curvature in the Euclidean space $\r^n$ $(n\ge 3)$. 
In particular, we establish the existence of such surfaces with the complex structure of any open   conformal non-orientable surface of finite topology and parabolic type.

\smallskip

\noindent{\bf Keywords}\hspace*{0.1cm} 
Minimal surface, non-orientable surface, Riemann surface, meromorphic function, complete surface of finite total curvature, Runge approximation.

\smallskip

\noindent{\bf Mathematics Subject Classification (2010)}\hspace*{0.1cm} 53A10, 53C42, 30D30, 32E30.

}



\section{Introduction}\label{sec:intro}  

A subject of longstanding interest in global minimal surface theory is the class of complete minimal surfaces of finite total curvature in $\R^n$  ($n \ge 3$); the literature is vast (we refer, among many others, to the surveys in \cite{Osserman1986,HoffmanKarcher1997,LopezMartin1999,MeeksPerez2012} or in \cite[Sec.\ 4.1]{AlarconForstnericLopez2021}). Classical structural theorems established by Huber \cite{Huber1957CMH} and Chern and Osserman \cite{ChernOsserman1967JAM} show that if an open Riemann surface $M$ carries a complete conformal minimal immersion of finite total curvature then $M$ is biholomorphic to $\Sigma \setminus E$, where $\Sigma$ is a compact Riemann surface and $E \subset \Sigma$ is a non-empty finite set (or, equivalently, $M$ has finite topology and parabolic type). Furthermore, the exterior derivative $dX$ of such an immersion $X:\Sigma\setminus E\to\mathbb{R}^n$ (which coincides with its $(1,0)$-part $\partial X$) is holomorphic and extends meromorphically to $\Sigma$ with an effective pole at every point in $E$. Such surfaces are proper in $\mathbb{R}^n$, as proven by Jorge and Meeks \cite{JorgeMeeks1983T}.

While orientable minimal surfaces have been intensively studied, immersed non-orientable minimal surfaces in $\mathbb{R}^n$ exhibit a significantly higher level of topological and analytical complexity. In this setting, the pioneering work of William H. Meeks III \cite{Meeks1981DMJ} laid the foundational geometric and conformal framework for complete non-orientable minimal surfaces of finite total curvature. Notably, Meeks constructed the first complete non-orientable minimal surface of finite total curvature in $\r^3$ (his well-known minimal M\"obius strip) and established the topological congruence formula governing such surfaces. Specifically, Meeks showed that an open non-orientable conformal surface $N$ can be investigated via its two-sheeted orientable conformal covering $\iota: M \to N$, where $M$ is an open Riemann surface equipped with a fixed-point-free antiholomorphic involution $\mathcal{I}: M \to M$ (the deck transformation of $\iota$) such that $N = M / \mathcal{I}$. Within this framework, every conformal minimal immersion $Z: N \to \mathbb{R}^n$ lifts to an $\mathcal{I}$-invariant conformal minimal immersion $X = Z \circ \iota: M \to \mathbb{R}^n$, that is, satisfying $X \circ \mathcal{I} = X$. Conversely, every non-orientable conformal minimal immersion $Z:N\to \r^n$ arises precisely through this quotient construction, being uniquely determined by a corresponding $\mathcal{I}$-invariant orientable minimal immersion $X:M\to \r^n$. Crucially, $Z: N \to \mathbb{R}^n$ is complete and of finite total curvature if and only if its orientable lift $X: M \to \mathbb{R}^n$ shares these properties; in this case, $M = \Sigma \setminus E$, where $\Sigma$ is a compact Riemann surface, $\mathcal{I}$ extends antiholomorphically without fixed points to $\Sigma$, and $E \subset \Sigma$ is an $\mathcal{I}$-invariant finite set ($\mathcal{I}(E) = E$).
 
Holomorphic approximation and interpolation theory is a classical and fundamental discipline in complex analysis, with deep ramifications in global differential geometry. Since the seminal works of Runge \cite{Runge1885AM}, Mittag-Leffler \cite{Mittag-Leffler1884AM}, and Mergelyan \cite{Mergelyan1951DAN} in the complex plane, and their generalizations to open Riemann surfaces by Behnke and Stein \cite{BehnkeStein1949}, Florack \cite{Florack1948}, Royden \cite{Royden1967JAM}, and Bishop \cite{Bishop1958PJM} (see also Fornæss, Forstnerič, and Wold \cite{FornaessForstnericWold2018} for a comprehensive survey), the capability to approximate continuous or holomorphic maps by global algebraic objects with exact control over zeros and poles has been a cornerstone of modern complex function theory. Via the classical Weierstrass representation formula and modern tools from Oka theory \cite{Forstneric2017}, these complex-analytic methods have been successfully imported into the global theory of conformal minimal surfaces in Euclidean space $\mathbb{R}^n$; see \cite{AlarconForstnericLopez2021} for a comprehensive monograph on the topic. In particular, approximation and interpolation theory for minimal surfaces has witnessed major progress in recent decades (see \cite[Ch.\ 3]{AlarconForstnericLopez2021}). 

In this direction, Alarc\'on and L\'opez \cite{AlarconLopez2022APDE} established general approximation results in arbitrary dimension $n \ge 3$ for complete orientable minimal surfaces of finite total curvature, including algebraic Runge-type and Mittag-Leffler-type theorems with jet interpolation and flux control; see \cite{Lopez2014TAMS,AlarconCastro-InfantesLopez2019CVPDE} for the case $n=3$. In the non-orientable setting, Alarc\'on, Forstneri\v c, and L\'opez \cite{AlarconForstnericLopez2020Memoirs} developed a comprehensive complex-analytic framework for uniform approximation of non-orientable minimal surfaces in $\mathbb{R}^n$ (see \cite{AlarconLopez2015GT} for the case $n=3$); however, these results apply to general open non-orientable surfaces without controlling the finite total curvature condition or the meromorphic behavior at the punctures. Extending these approximation results to complete non-orientable surfaces of finite total curvature posed deep technical challenges: the requirement to preserve the $\mathcal{I}$-invariance of the Weierstrass data while simultaneously controlling periods, divisors, poles, and metric completeness precluded the direct application of existing methods.   

The main purpose of this paper is to resolve this open problem. 
More precisely, we establish the following Runge-Mittag-Leffler-type uniform approximation theorem with jet interpolation (also at some ends) and full flux control for complete non-orientable conformal minimal surfaces of finite total curvature in $\mathbb{R}^n$ ($n \ge 3$) with a prescribed complex structure. 
%
%
\begin{theorem}\label{th:intro}
Let $\Sigma$ be a compact Riemann surface endowed with a fixed-point-free antiholomorphic involution $\mathcal{I}:\Sigma\to \Sigma$, and let $\varnothing\neq E\subset \Sigma$ be a finite $\mathcal{I}$-invariant subset. Also let $K\subset \Sigma\setminus E$ be an $\mathcal{I}$-invariant, smoothly bounded, Runge compact domain (possibly disconnected) and $E'$ and $\Lambda$ be a pair of disjoint (possibly empty) $\mathcal{I}$-invariant finite sets in $\mathring K$.
If $X: K\setminus E'\to \mathbb{R}^n$ ($n\geq 3$) is an $\mathcal{I}$-invariant complete conformal minimal immersion of finite total curvature, then for any $\epsilon>0$, any integer $r\ge 0$, and any group homomorphism $\mathfrak{F}:H_1(\Sigma\setminus(E\cup E'),\mathbb{Z})\to\mathbb{R}^n$ such that $\mathfrak{F}(\gamma)=\mathrm{Flux}_X(\gamma)$ for every closed curve $\gamma\subset K\setminus E'$ and $\mathfrak{F}(\mathcal{I}\circ\gamma)=-\mathfrak{F}(\gamma)$ for every closed curve $\gamma\subset \Sigma\setminus(E\cup E')$, there is an $\mathcal{I}$-invariant conformal minimal immersion $Y: \Sigma\setminus (E\cup E')\to \mathbb{R}^n$ satisfying the following conditions.
\begin{enumerate}[\rm (i)]
\item $Y$ is complete and of finite total curvature.
\smallskip
\item $Y-X$ extends harmonically to $K$ and $|Y-X|<\epsilon$ everywhere on $K$.
\smallskip
\item $Y-X$ vanishes at least to order $r$ at every point of $E'\cup \Lambda$.
\smallskip
\item $\mathrm{Flux}_Y(\gamma)=\mathfrak{F}(\gamma)$ for every closed curve $\gamma$ in $\Sigma\setminus (E\cup E')$.
\end{enumerate}
\end{theorem}
Theorem \ref{th:intro} is new in every dimension, including the case $n=3$. Although formulated in terms of two-sheeted orientation covers and $\mathcal{I}$-invariant immersions, this result is inherently a theorem for non-orientable surfaces. 
We point out that it can be strengthened by guaranteeing Mergelyan-type approximation rather than just Runge-type approximation, as is clear from the proof. We do not pursue this here for simplicity of exposition. We also note that our proof slightly simplifies in the particular case when $n=3$, as is clearly seen in the proof of Lemma \ref{lem:core}.

Few examples of complete non-orientable minimal surfaces of finite total curvature are available in the literature (see for instance  \cite{LopezMartin1999}).
As an immediate corollary of Theorem \ref{th:intro}, we record that every open conformal non-orientable surface of finite topology and parabolic type is the underlying complex structure of a complete non-orientable minimal surface of finite total curvature in Euclidean space. The following existence result is also new even in case $n=3$.
%
%
\begin{corollary}
Given $\Sigma$, $\mathcal{I}$, $E$, and $n$ as in Theorem \ref{th:intro}, there exists a complete $\mathcal{I}$-invariant full conformal minimal immersion $\Sigma\setminus E\to \mathbb{R}^n$ of finite total curvature.
\end{corollary}
The analogue of Theorem \ref{th:intro} in the orientable setting was established in \cite[Theorem 1.2]{AlarconLopez2022APDE}. We emphasize that the proof strategy in \cite{AlarconLopez2022APDE} is fundamentally incompatible with preserving $\mathfrak{I}$-invariance, and thus cannot be adapted to the non-orientable case. (An alternative proof of \cite[Theorem 1.2]{AlarconLopez2022APDE} in the case $E'=\varnothing$ was obtained by Alarc\'on and L\'arusson in \cite{AlarconLarusson2025Crelle}, where a different method is used to establish analogous results for the more general family of regular immersions into $\C^n$ directed by algebraically elliptic cones; see \cite{AlarconForstnericLarusson2026JLMS} for further developments in this direction. This method, relying on algebraic Oka theory, also does not seem to adapt to the non-orientable framework.) Our approach here is substantially different, relying on a genuine algebraic approximation strategy specifically tailored to the $\mathfrak{I}$-invariant framework.

As in \cite{AlarconLopez2022APDE}, the proof of Theorem \ref{th:intro} in Section \ref{sec:proof} combines the technique of sprays for controlling periods with classical algebraic approximation for meromorphic functions on compact Riemann surfaces. A main novelty is an approximation result for $\mathfrak{I}$-invariant meromorphic $1$-forms with values in the null quadric $\mathfrak{A}_*^{n-1}$ \eqref{eq:nullquadric} directing minimal surfaces, which we establish in Section \ref{sec:core} (see Lemma \ref{lem:core}). This enables us to approximate $\partial X$ by an $\mathfrak{I}$-invariant meromorphic $1$-form $\phi$ on $\Sigma$ taking values in $\mathfrak{A}_*^{n-1}$, which is holomorphic on $\Sigma\setminus (E\cup E')$ with a pole at every point in $E\cup E'$, and such that $\phi-\partial X$ is holomorphic on $K$. 
With this in hand, the main task in the proof is to deform $\phi$ within $\mathfrak{A}_*^{n-1}$ to ensure that its real part is exact, the primary complication being to prevent the appearance of unwanted zeros without losing the $\Igot$-invariance. This requires a substantially more involved construction of period-dominating sprays through successive approximation stages (see Steps 2 to 5 in the proof).


\section{Preliminaries and notation}\label{sec:prelim}

We write $\n=\{1,2,3,\ldots\}$, $\z_+=\n\cup\{0\}$, and $\imath=\sqrt{-1}$, and denote by $\cp^1=\c\cup\{\infty\}$ the Riemann sphere. We shall use the symbols $\Re$ and $\Im$ to denote, respectively, the real and the imaginary part, and identify $\c^n$ with $\r^{2n}$. 
We denote by $|\cdot|$ the Euclidean norm in $\r^n$.
The uniform norm (or $\sup$ norm) of  a map $f\colon X\to \r^n$ on a set $X$ is  
\[
\|f\|_{X}=\sup  \{|f(x)|\colon x\in X\}\in [0,+\infty].
\]
If $f,g\colon X\to \r^n$ are maps,   the notation $f\approx g$ shall mean that $\|f-g\|_X$ is so close to $0$   that no significant deviation  between $f$ and $g$  can be found  in a given argumentation. In this case we say that $f$ approximates $g$ on $X$.
\begin{definition}\label{de:full}
Given a set $X$, a map   $X\to \c^n$  is said to be {\em full} if its image lies in no affine hyperplane of $\C^n$. Full maps $X\to\cp^n$ are defined in the same way. 
\end{definition}
Assume that $X$ is a topological space.  A {\em Jordan arc} in  $X$ is an embedding $ [0,1]\to X$; an open Jordan arc in $X$ is an embedding  $(0,1)\to X$.  Continuous maps $C\colon \s^1\to X$ are said to be closed curves in $X$;  if in addition $C$ is an embedding then  the closed curve  is said to be simple or a {\em Jordan curve}. Usually we shall identify arcs and curves with their image.  We denote by $H_1(X,\z)$ the first homology group with integer coefficients on $X$.  

A smooth surface is said to be {\em open} if it is not compact and has no boundary. Throughout this paper surfaces are considered to have no boundary unless the contrary is indicated.
Assume that $M$ is an (open) smooth surface.
%
%
\begin{definition}\label{de:Runge}
A nonempty (possibly disconnected) compact set $S$ in  $M$ is said to be {\em Runge} if $M\setminus S$ has no relatively compact components.
\end{definition}

\subsection{Divisors and function spaces}\label{ss:spaces}
Given a  set $X$, we denote by $\Div(X)$ the free commutative group of {\em finite} divisors of $X$ with multiplicative notation:
\[
	\Div(X)=\Big\{\prod_{j=1}^k q_j^{n_j}\colon k\in \n,\; q_j\in X,\; n_j\in \z \Big\}.
\]
Here,  $q^0=1$ for all $q\in X$. 
Given $D=\prod_{j=1}^k q_j^{n_j}\in\Div(X)$,  the set $\supp  (D)=\{q_j\colon n_j\neq 0\}\subset X$ is said to be the {\em support} of $D$. The divisor $D$ is said to be {\em effective} if $n_j\geq 0$ for all $j=1,\ldots,k$. We write $D_1\geq D_2$ to mean that $D_1 D_2^{-1}$ is   effective.  

Let  $M$ and $N$ be a pair of complex manifolds and $S\subset M$ be a subset. We denote by $\Ccal^0(S,N)$   the space of continuous maps $S\to N$, and  write  $\Ccal^0(S)=\Ccal^0(S,\c)$. As it is customary, we denote by $\Ocal(S,N)$ the space of all holomorphic maps from some neighborhood of $S$ in $M$ (depending on the function) into $N$.  

We assume in the sequel that $M$ is a Riemann surface (either open or compact). For any subset $S\subset M$ we denote $\Ocal(S)=\Ocal(S,\c)$, whereas $\Ocal_\infty(S)$ will denote  the space of all meromorphic functions on some neighborhood of $S$ in $M$.   For a finite subset $E\subset \mathring S$, we denote 
\begin{equation}\label{eq:OcalinftySE}
	 \Ocal_\infty(S|E)=\Ocal_\infty(S)\cap \Ocal(S\setminus E);
\end{equation}
i.e., $\Ocal_\infty(S|E)$ is the space of all meromorphic functions on a neighborhood of $S$ which have poles (if any) only at points in $E$.
Likewise, we denote by $\Omega(S)$  the space of all holomorphic $1$-forms on some neighborhood of $S$ in $M$, $\Omega_\infty(S)$  the space of all meromorphic $1$-forms on some neighborhood of $S$ in $M$, and   $\Omega_\infty(S| E)=\Omega_\infty(S)\cap\Omega(S\setminus E)$.

Assume that the set $S\subset M$ is compact. For any $f\in \Ocal_\infty(S)$, $f\not\equiv 0$, $f\not\equiv\infty$, we call $Z(f)$ and $ P(f)$   the (finite) sets of zeros and poles of $f$ in $S$, and write $s_p\in \n$  for the zero or pole order of $f$ at $p$ for all $p\in Z(f)\cup P(f)$ . We set 
 \[
 	[f]_0=\prod_{p\in Z(f)} p^{s_p}\quad \text{and}\quad [f]_\infty=\prod_{p\in P(f)} p^{s_p}
\]
the effective divisors in $\Div(S)$ of zeros and poles of $f$ in $S$, respectively, and   
\[
	[f]:=\frac{[f]_0}{[f]_\infty}\in \Div(S)
\]  the divisor of $f$ in $S$. (We do not use the more customary parenthetical notation for divisors in order to avoid ambiguities.) We use the same notation for the corresponding divisors of a nonzero meromorphic $1$-form on $S$. 

For each effective divisor $D\in \Div(S)$   we set
\begin{equation}\label{eq:OcalDS}
	\Ocal_{D}(S)=\{f\in   \Ocal(S)\colon  [f]\geq D\}.
\end{equation}

Let $n\ge 3$ be an integer. We denote by
 \begin{equation}\label{eq:nullquadric1}
	\Sgot^{n-1}_*=\Big\{(u_1\ldots,u_n)\in \c^n\setminus\{0\}\colon u_1 u_2 = \sum_{j=3}^n u_j^2\Big\}.
\end{equation}
The punctured complex quadric $\Sgot^{n-1}_*$ is canonically identified with the punctured null quadric
 \begin{equation}\label{eq:nullquadric}
\Agot^{n-1}_*=\Big\{(z_1\ldots,z_n)\in \c^n\setminus\{0\}\colon \sum_{j=1}^n z_j^2=0\Big\}
\end{equation}
by the natural linear biholomorphism  $\Xi \colon \c^n\setminus\{0\}\to \c^n\setminus\{0\}$ given by
 \begin{equation}\label{eq:Thetabiho}
  \Xi(z_1,z_2,\ldots,z_n):=\big(-z_1+ \imath z_2,z_1+\imath z_2,z_3,\ldots, z_n\big),
   \end{equation}
which maps $\Agot^{n-1}_*$ into $\Sgot^{n-1}_*$. The punctured null quadric $\Agot_*^{n-1}$ is a complex homogeneous manifold, and hence an Oka manifold (see \cite[Example 5.6.2]{Forstneric2017}).

For any map $f=(f_1,\ldots,f_n)\in\Ccal^0(S,\cp^1)^n$ write
\begin{equation}\label{eq:finfi}
f^{-1}(\infty)=\bigcup_{j=1}^n f_j^{-1}(\infty),
\end{equation}
 and denote
 \[
\Ocal_\infty(S,\Sgot^{n-1}_*)=\{f\in \Ocal_\infty(S)^n\colon f\big(S\setminus f^{-1}(\infty)\big)\subset \Sgot^{n-1}_*\}.
\]
  Given a finite set $E\subset\mathring S$,  we denote
  \[
  \Ocal_\infty(S|E,\Sgot^{n-1}_*)=  \Ocal_\infty(S,\Sgot^{n-1}_*)\cap   \Ocal(S\setminus E,\Sgot^{n-1}_*).
  \] 
   The spaces $\Ocal_\infty(S,\Agot^{n-1}_*)$ and   $\Ocal_\infty(S|E,\Agot^{n-1}_*)$  are defined in the same way. Likewise, the spaces of   $1$-forms   $\Omega(S,\Sgot_*^{n-1})$, $\Omega_{\infty}(S,\Sgot_*^{n-1})$ and $\Omega_{\infty}(S|E,\Sgot_*^{n-1})$ with values in $\Sgot_*^{n-1}$, as well as the ones with values in $\Agot_*^{n-1}$, are defined in the obvious way. In particular, $\Omega_{\infty}(S|E,\Sgot_*^{n-1})$ consists of the meromorphic $1$-forms   $(\phi_1,\ldots,\phi_n)\in \Omega_{\infty}(S|E)^n$ such that $\phi_1 \phi_2 = \sum_{j=3}^n \phi_j^2$.
   
   Passing to the non-orientable framework, assume in addition that $\Igot:M\to M$ is a fixed-point-free antiholomorphic involution. A subset $S\subset M$ is $\Igot$-invariant if $\Igot(S)=S$. If $S\subset M$ is $\Igot$-invariant, a map $X:S\to \r^n$ is $\Igot$-invariant if $X\circ \Igot=X$, a map $f:S\to \c^n$ is $\Igot$-invariant if $f\circ \Igot=\bar f$, and a $\c^n$-valued meromorphic  $1$-form $\phi$ on $S$ is $\Igot$-invariant if $\Igot^*\phi=\bar \phi$. Finally, we set
$\Ocal_\Igot(S)=\{f\in \Ocal(S)\colon f \text{ is $\Igot$-invariant}\}$ and $\Omega_\Igot(S)=\{\phi\in \Omega(S) \colon \phi \text{ is $\Igot$-invariant}\}$. Likewise we define $\Ocal_{\infty,\Igot}(S)$,  $\Omega_{\infty,\Igot}(S)$,  and for any $\Igot$-invariant finite set $E\subset \mathring S$, $\Ocal_{\infty,\Igot}(S|E)$ and $\Omega_{\infty,\Igot}(S|E)$, for the corresponding spaces of meromorphic objects. A divisor $D=\prod_{j=1}^k q_j^{n_j}\in\Div(S)$ is $\Igot$-invariant if $\Igot(D):=\prod_{j=1}^k \Igot(q_j)^{n_j}$ equals $D$, we denote by $\Div_\Igot(S)$ the set of all such divisors. If $D\in \Div_\Igot(S)$,  we set $\Ocal_{D,\Igot}(S)=\Ocal_{D}(S)\cap \Ocal_\Igot(S)$. Finally, we set
\[
\Ocal_{\infty,\Igot}(S,\Sgot^{n-1}_*)=\{f\in\Ocal_{\infty}(S,\Sgot^{n-1}_*)\colon  f_2=-\bar f_1\circ \Igot, \, f_j\in \Ocal_{\infty,\Igot}(S), \,3\leq j\leq n \},
\]
and
\[
\Ocal_{\infty,\Igot}(S,\Agot^{n-1}_*)= \Ocal_{\infty}(S,\Agot^{n-1}_*) \cap \Ocal_{\infty,\Igot}(S)^n.
\]
Observe that $\Xi(\Ocal_{\infty,\Igot}(S,\Agot^{n-1}_*))=\Ocal_{\infty,\Igot}(S,\Sgot^{n-1}_*)$.
The spaces $ \Ocal_{\infty,\Igot}(S|E,\Sgot^{n-1}_*)$ and  $ \Ocal_{\infty,\Igot}(S|E,\Agot^{n-1}_*)$, as well as  $ \Omega_{\Igot}(S,\Agot^{n-1}_*)$, $ \Omega_{\Igot}(S,\Sgot^{n-1}_*)$, 
$ \Omega_{\infty,\Igot}(S,\Agot^{n-1}_*)$, $ \Omega_{\infty,\Igot}(S,\Sgot^{n-1}_*)$,
 $ \Omega_{\infty,\Igot}(S|E,\Sgot^{n-1}_*)$ and  $ \Omega_{\infty,\Igot}(S|E,\Agot^{n-1}_*)$, are defined analogously.


\subsection{$\Igot$-invariant minimal surfaces in $\r^n$}\label{sec:minimal}
Let $M$ be an open Riemann surface.
 A map  $X=(X_1,\ldots,X_n)\colon M\to\r^n$ $(n\ge 3)$  is a  conformal minimal immersion if and only if  $X$ is harmonic, and its complex derivative $\di X\in\Omega(M)^n$ (i.e., the $(1,0)$-part of the exterior differential $dX$ of $X$) vanishes nowhere on $M$ and satisfies $\sum_{j=1}^n\di X_j^2\equiv 0$. Given a holomorphic $1$-form $\theta$ on $M$ with no zeros, the last two conditions are equivalent to $\di X/\theta\in \Ocal(M,\Agot^{n-1}_*)$ or to $\Xi\big( \di X/\theta\big)\in \Ocal(M,\Sgot^{n-1}_*)$; see \eqref{eq:Thetabiho}. Moreover, in this case $X\colon M\to\r^n$ is given by the formula
\[
	M\ni p\longmapsto X(p_0)+\Re\int_{p_0}^p 2\di X
\]
for any fixed base point $p_0\in M$. This is known in the literature as the {\em Enneper-Weierstrass representation} formula for minimal surfaces in $\r^n$ (see e.g.\ \cite{Osserman1986}).
 A map $X:S\to \r^n$ from a subset $S\subset M$ is said to be a {\em conformal minimal immersion} if it extends as a conformal minimal immersion to some open neighborhood of $S$ in $M$ (depending on $X$).

%
%
%
\begin{definition}\label{def:full-CMI}
We say that a conformal minimal immersion $X:M\to\r^n$ is {\em full} if 
the holomorphic map $\di X/\theta\colon M\to\c^n$ is full for any holomorphic $1$-form $\theta$ vanishing nowhere on $M$.
\end{definition}
The group homomorphism $\Flux_X\colon H_1(M,\z)\to\r^n$ given by
\[
	\Flux_X(\gamma)=2\int_\gamma\Im (\di X)=-2\imath\int_\gamma\di X
	\quad
	\text{for every loop $\gamma\subset M$},
\]
is said to be the {\em flux map} (or just the {\em flux}) of $X$. 
A conformal minimal immersion $X:M\to \r^n$  is said to be of {\em finite total curvature} (acrostically, FTC) if
\[
\int_M K\, dA>-\infty,
\]
where $K$ is the Gauss curvature of the Riemannian metric 
\begin{equation}\label{eq:ds2}
	ds^2=2\sum_{j=1}^n |\di X_j|^2
\end{equation}
induced on $M$ by the Euclidean one via $X$, and $dA$ is the area element of $ds^2$. 

The conformal minimal immersion $X:M\to\r^n$ is called {\em complete} if the metric $ds^2$ is complete in the classical sense of Riemannian geometry. Assume now that $M$ is a Riemann surface with compact boundary $bM\subset M$ (possibly $bM=\varnothing$). If $M$ carries a {\em complete} conformal minimal immersion $X:M\to \r^n$ of FTC (here, $X$ is complete if and only if  $X\circ \gamma$ has infinite Euclidean length for any divergent arc $\gamma\colon[0,1)\to M$), then the classical results by Huber \cite{Huber1957CMH} and Chern-Osserman \cite{ChernOsserman1967JAM} impose the following conditions.
\begin{enumerate}[\rm (i)]
\item $M$ is conformally equivalent to $R\setminus E$, where $R$ is  a compact Riemann surface with compact boundary $bR$ and $E\subset R\setminus bR$ is a finite subset.
\smallskip
\item $\di X$ extends meromorphically to $R$ with an effective pole at each point of $E$. 
\end{enumerate}
Conversely, if $R$ and $E$ are as in {\rm (i)} and if $X\colon R\setminus E \to \r^n$ is a conformal minimal immersion satisfying {\rm (ii)}, then $X$ is complete and of FTC.

If an open Riemann surface $M$ is endowed with a fixed-point-free antiholomorphic involution $\Igot$, $\theta$ is a nowhere vanishing  holomorphic $\Igot$-invariant  $1$-form on $M$ (see \cite[Corollary 6.5]{AlarconLopez2015GT}),  and $X:M\to \r^n$ is an $\Igot$-invariant conformal minimal immersion, then $\di X/\theta\in \Ocal_\Igot(M,\Agot^{n-1}_*)$ and  $\Xi\big( \di X/\theta\big)\in \Ocal_\Igot(M,\Sgot^{n-1}_*)$; see \eqref{eq:Thetabiho}. Furthermore, $\Flux_X(\Igot\circ \gamma)=-\Flux_X(\gamma)$ for every loop $\gamma\subset M$. Assume now  that $M$ is a Riemann surface with compact boundary $b M\subset M$ and $X:M\to \r^n$ is an $\Igot$-invariant complete conformal minimal immersion of FTC. Then $M=R\setminus E$ as above,   $\Igot$ extends antiholomorphically and without fixed points to $R$, and $\Igot(E)=E$.


\section{Approximation of $\Igot$-invariant meromorphic $1$-forms with values in the null quadric}\label{sec:core}

The following result will be crucial in the proof of Theorem \ref{th:intro} in the next section.
%
%
\begin{lemma}\label{lem:core}
Let $\Sigma$ be a compact Riemann surface (without boundary) endowed with a fixed-point-free antiholomorphic involution $\Igot:\Sigma\to \Sigma$, $E\subset \Sigma$ an $\Igot$-invariant finite set,  $K\cup \subset \Sigma\setminus E$    an $\Igot$-invariant smoothly bounded compact domain that is a strong deformation retract of $\Sigma\setminus E$, and $E',\Lambda\subset \mathring K$ a pair of disjoint $\Igot$-invariant finite sets. 
Then every  $\Phi\in \Omega_{\infty,\Igot}(K|E',\Sgot_*^{n-1})$ can be  approximated uniformly on $K$ by meromorphic $1$-forms  $\Psi\in \Omega_{\infty,\Igot}(\Sigma|E'\cup E,\Sgot_*^{n-1})$ such that $\Psi$ has an effective pole at every point in $E$ and $\Psi-\Phi$ is holomorphic on $K$   and   vanishes to any given finite order $r$ at every point in $E'\cup \Lambda$.
\end{lemma}
Here the approximation means that $\Psi/\theta$ is close to $\Phi/\theta$ uniformly on $K$ for any given $\theta\in  \Omega(\Sigma\setminus E)$ without zeroes.
%
%
\begin{proof} 
Write $E=E_0\cup \Igot(E_0)$ for some $E_0\subset E$ with $E_0\cap \Igot(E_0)=\varnothing$. We begin with the following claim.
\begin{claim}\label{cl:theta}
There is an $\Igot$-invariant meromorphic $1$-form $\theta$ on $\Sigma$ that is holomorphic on $\Sigma\setminus E$, has an effective pole at every point in $E$, and vanishes nowhere on $K$.
\end{claim}
\begin{proof}  Choose $\tau\in \Omega_{\infty,\Igot}(\Sigma|E)$
with  an effective pole at every point in $E$; for instance, we can take $\tau=\tau_0+\Igot^*(\bar \tau_0)$ for any $\tau_0\in \Omega_{\infty}(\Sigma|E_0)$ with an effective pole at every point in $E_0$ (see \cite[Claim 6.3]{AlarconLopez2022APDE}). Let $D\cdot \Igot(D)$ be the divisor of $\tau|_K$, where $\supp(D)\cap \supp(\Igot(D))=\varnothing$. By    \cite[Proposition 4.1]{AlarconLopez2022APDE} there is $f\in \Ocal_\infty(\Sigma|E_0\cup \supp(D))$ vanishing nowhere on $K$, having an effective pole at every point in $E_0$, and such that $f\tau$ is holomorphic and nowhere vanishing on $\supp(D)$. Then $(\bar f\circ \Igot) f \tau$ solves the claim.
\end{proof}

Let $\theta\in \Omega_{\infty,\Igot}(\Sigma|E)$ be a $1$-form given by Claim \ref{cl:theta} and denote by $Z\cup \Igot(Z)\subset \Sigma\setminus (K\cup E)$ the zero set of $\theta$ in $\Sigma$, where $Z\cap \Igot(Z)=\emptyset$. Set $f=(f_1,\ldots,f_n)=\Phi/\theta\in \Ocal_{\infty,\Igot}(K|E',\Sgot_*^{n-1})$.  
Choose a large integer  $r_0>r$ to be specified later. 

We claim that $f_1$ can be approximated uniformly on $K$ by a function $\tilde f_1\in  \Ocal_{\infty}(\Sigma|E\cup E'\cup Z\cup \Igot(Z))$ such that:
\begin{enumerate}[\rm ({a}1)]
\item    $\tilde f_1$ has an effective pole at every point in $E$.
\item $\tilde f_1-f_1\in \Ocal(K)$ and vanishes to order $r_0$ at every point in $E'\cup \Lambda$.
\item   $[\tilde f_1|_K]=[f_1]$.
\item $\tilde f_1 \theta\in  \Omega_{\infty}(\Sigma|E\cup E') $ and $[\tilde f_1 \theta|_{\Sigma\setminus (E\cup K)}]=D^2$ for some integer divisor $D\geq 1$ such that $\supp (D)\cap (Z\cup \Igot(Z))=\varnothing$.
\item $\supp (D)\cap \Igot(\supp(D))=\varnothing$.
\end{enumerate}
Indeed, by   \cite[Proposition 4.1]{AlarconLopez2022APDE} we can approximate $f_1$ uniformly on $K$ by  a function $g_1\in \Ocal_{\infty}(\Sigma|E_0\cup E'\cup Z\cup \Igot(Z))$ 
satisfying (a1) to (a4) with $g_1$ in place of $\tilde f_1$, but perhaps not (a5). (It suffices to apply the proposition to a suitable extension of $f_1$ that is  meromorphic  on a  neighborhood of $Z\cup\Igot(Z)$. Take also into account Hurwitz theorem for (a3).) We thus have by (a4) that $[g_1 \theta|_{\Sigma\setminus (E\cup K)}]=A^2$ for some  integer divisor $A\geq 1$ such that $\supp (A)\cap (Z\cup \Igot(Z))=\varnothing$.  Since the smoothly bounded compact domain $K$ is a strong deformation retract of $\Sigma\setminus E$, we have that $\Sigma\setminus K=C\cup\Igot(C)$, where $C$ is a disjoint union of finitely many open discs with $C\cap\Igot(C)=\varnothing$ and each component of $C$ containing a single point of $E_0\subset C$. Observe that $K\cup C$ is a Runge compact strong deformation retract of $\Sigma \setminus \Igot(E_0)$. As above,   \cite[Proposition 4.1]{AlarconLopez2022APDE} enables us to  approximate $g_1|_{K\cup C}$ uniformly on $K\cup C$ by a function $\tilde f_1\in  \Ocal_{\infty}(\Sigma|E\cup E'\cup Z\cup\Igot(Z))$ satisfying (a1) to (a4) with    $[\tilde f_1 \theta|_C]=[g_1 \theta|_C]$ and $\tilde f_1$ vanishes nowhere in $\supp(A)\cap \Igot(C)$. Observe that $\tilde f_1$ satisfies (a5) as well.
 
 Set 
 \[
 \tilde f_2=-\overline{\tilde f_1\circ \Igot}\in  \Ocal_{\infty}(\Sigma|E\cup E'\cup Z\cup \Igot(Z)).
 \]
Since $f_2=-\bar f_1\circ \Igot$, we have that $\tilde f_2$  is close to $f_2$ uniformly on $K$ and the following properties hold.
\begin{enumerate}[\rm ({b}1)]
\item    $\tilde f_2$ has an effective pole at every point in $E$.
\item $\tilde f_2-f_2\in \Ocal(K)$ and vanishes to order $r_0$ at every point in $E'\cup \Lambda$.
\item   $[\tilde f_2|_K]=[f_2]$.
\item $\tilde f_2 \theta\in  \Omega_{\infty}(\Sigma|E\cup E') $ and $[\tilde f_2 \theta|_{\Sigma\setminus (E\cup K)}]=\Igot(D)^2$.
\end{enumerate}
In particular, (a4), (a5), and (b4) show that    
\begin{equation}\label{eq:common}
\text{$\tilde f_1\theta$ and $\tilde f_2\theta$ have no common zeroes in $\Sigma\setminus (E\cup K)$.}
\end{equation}

Assuming that the approximations in (a1) and (b1) are close enough,  we claim   that each $f_j$, $3\leq j\leq n$, can be approximated uniformly on $K$ by a function $\tilde f_j\in  \Ocal_{\infty,\Igot}(\Sigma|E\cup E'\cup Z\cup \Igot(Z))$ such that:
\begin{enumerate}[\rm ({c}1)]
\item $\tilde f_j-f_j\in \Ocal(K)$ and vanishes to order $r$ at every point in $E'\cup \Lambda$ and to order $r_0$ on $ \supp([(f_1 f_2)|_K]_0)$.
\item   $[\tilde f_j|_K]=[f_j]$.
\item $\tilde f_j \theta\in  \Omega_{\infty,\Igot}(\Sigma|E\cup E') $ and $\tilde f_j \theta$ vanishes nowhere on $Z\cup \Igot(Z)$. 
\item If $n=3$ then $\tilde f_3^2= \tilde f_1 \tilde f_2$.
\item If $n\geq 4$ then $\tilde f_j \theta$ vanishes to order $r_1$ at every point in $ \supp([\tilde f_1 \tilde f_2]_0)\setminus K$ for some $r_1\in \n$ to be specified later.
\end{enumerate}
Indeed, if $n=3$ it is clear that $\tilde f_1 \tilde f_2\in \Ocal_{\infty,\Igot}(\Sigma|E\cup E'\cup Z\cup \Igot(Z))$ by the definition of $\tilde f_2$. Furthermore, taking into account that $f_3^2=f_1 f_2$ on $K$, we have by (a3), (a4), (b3), and (b4)  that 
$[(\tilde f_1 \tilde f_2)|_{\Sigma\setminus E}]=[f_3|_K]^2 D^2 \Igot (D)^2 [\theta]_0^{-2}$. Using that $K$ is a strong deformation retract of $\Sigma\setminus E$ and assuming that the approximation in (a1) and (b1) are sufficiently close, we deduce that $\tilde f_1 \tilde f_2=h^2$ for some $h\in  \Ocal_{\infty,\Igot}(\Sigma|E\cup E'\cup Z\cup \Igot(Z))$ close to $f_3$ uniformly on $K$. Observe that  $\tilde f_3=h$ satisfies conditions (c1) to (c4) for $j=n=3$; for (c1) choose $r_0>r$ sufficiently large.

If $n\geq 4$, fix $j\in \{3,\ldots,n\}$. Arguing  as above,     \cite[Proposition 4.1]{AlarconLopez2022APDE} allows to  approximate $f_j$ uniformly on $K$ by  a function $g_j\in \Ocal_{\infty}(\Sigma|E_0\cup E'\cup Z)$ satisfying (c1), (c2) and (c5), with $g_j$ in place of $\tilde f_j$, such that 
 $g_j \theta\in  \Omega_{\infty}(\Sigma|E\cup E') $ and $g_j \theta$ vanishes nowhere on $Z$. Note that  $g_j \theta$ vanishes  at every point in $\Igot(Z)$. Since 
 $f_j=\bar f_j\circ \Igot$, if the approximation of $f_j$ by $g_j$ on $K$ is close enough then $\tilde f_j=\frac12(g_j+\bar g_j\circ \Igot)\in \Ocal_{\infty,\Igot}(\Sigma|E\cup E'\cup Z\cup \Igot(Z))$
is close to $f_j$ uniformly on $K$ and satisfies conditions (c1) to (c3) and (c5); for (c2) use Hurwitz theorem.

If $n=3$ then $\Psi=(\tilde f_1,\tilde f_2,\tilde f_3)\theta\in \Omega_{\infty,\Igot}(\Sigma|E'\cup E,\Sgot_*^{2})$ solves the lemma. 
In the sequel we assume that $n\geq 4$. Set 
\begin{equation}\label{eq:Delta}
\tilde \eta_j=\tilde f_j \theta \text{ for } j=1,2,\quad \tilde \phi_j=\tilde f_j \theta \text{ for } j=3,\ldots,n, \quad \Delta=\sum_{j=3}^n \tilde \phi_j^2.
\end{equation}
By Hurwitz theorem, (a3), (a4), (b3), (b4), (c1), (c2), and (c5), we have that
\[
[\Delta|_{K}]= [(\tilde \eta_1\tilde \eta_2)|_K]
\]
and
\[
[\Delta|_{\Sigma\setminus (K\cup E)}]= [(\tilde \eta_1\tilde \eta_2)|_{\Sigma\setminus (K\cup E)}] B \Igot(B)=D^2\Igot(D)^2B \Igot(B)
\]
for some integer divisor $B\in \Div(\Sigma\setminus (K\cup E))$ with 
\begin{equation}\label{eq:soporte0}
\supp(B)\cap \supp(\Igot(B))=\varnothing,
\end{equation}
 whenever the approximation of each $f_j$ by $\tilde f_j$, $j\geq 3$, is close enough and $r_1$ is chosen large enough.  Further, by (a5), we can choose $B$ so that
\begin{equation}\label{eq:soporte1}
	\supp(B)\cap \supp(\Igot(D))=\supp(B)\cap \supp([\tilde \eta_2|_{\Sigma\setminus (K\cup E)}])=\varnothing,
\end{equation}
hence
\begin{equation}\label{eq:soporte2}
	\supp(\Igot(B))\cap \supp(D)=\supp(\Igot(B))\cap \supp([\tilde \eta_1|_{\Sigma\setminus (K\cup E)}])=\varnothing.
\end{equation}
Define 
\begin{equation}\label{eq:FDelta}
F=\frac{\Delta}{\tilde \eta_1\tilde \eta_2}\in \Ocal_{\infty,\Igot}(\Sigma|E),
\end{equation}
 and observe that $F$ is close to $1$ uniformly on $K$ with $[F|_{\Sigma\setminus E}]=B \Igot(B)$. Reasoning as in the construction of $\tilde f_1$ above, an application of \cite[Proposition 4.1]{AlarconLopez2022APDE} in two steps provides a function $h_1\in \Ocal_\infty(\Sigma|E)$ close to $1$ uniformly on $K$ with $[h_1]_0=B B_1^2$, for some integer divisor $B_1\in \Div(\Sigma\setminus (K\cup E))$ with 
 \begin{equation}\label{eq:soporte3}
 \supp (B_1)\cap \supp (\Igot(B_1))=\varnothing \quad \text{and}\quad  \supp (B_1 \Igot(B_1))\cap \supp ([\tilde \eta_1\tilde \eta_2])=\varnothing;
 \end{equation}
 cf.\ (a5). Furthermore, we choose $h_1$ agreeing with $1$ to a high order $r_2>r\in \n$ at every point in $E'\cup \Lambda$ and having an effective pole at every point of $E$. Set $h_2=\bar h_1\circ \Igot\in \Ocal_\infty(\Sigma|E)$, and note that $h_2$ is close to $1$ uniformly on $K$ with $[h_2]_0=\Igot(B)\Igot( B_1)^2$. 
Note that \eqref{eq:soporte1}, \eqref{eq:soporte2}, and \eqref{eq:soporte3} guarantee that
 \begin{equation}\label{eq:soporte4}
 	\supp([h_1]_0)\cap \supp ([\tilde \eta_2]_0)=\varnothing
	\quad\text{and}\quad 
	\supp([h_2]_0)\cap \supp ([\tilde \eta_1]_0)=\varnothing.
 \end{equation}

  It follows that $h_1 h_2/F\in \Ocal_{\infty,\Igot}(\Sigma|E)$ is close to $1$ uniformly on $K$ with $[(h_1 h_2/F)|_{\Sigma\setminus E}]=B_1^2\Igot(B_1)^2$.
Since $K$ is a strong deformation retract of $\Sigma\setminus E$, assuming that the approximations are close enough, we have that 
\begin{equation}\label{eq:H2}
\frac{h_1 h_2}{F}=H^2 \quad \text{for some $H\in  \Ocal_{\infty,\Igot}(\Sigma|E)$.}
\end{equation}
Define
 \[
 \tau_j=h_j \tilde \eta_j \text{ for } j=1,2, \quad \psi_j=H \tilde \phi_j \text{ for } j=3,\ldots,n. 
 \]
It turns out that $\Psi=(\tau_1,\tau_2,\psi_3,\ldots,\psi_n)\in \Omega_{\infty,\Igot}(\Sigma|E'\cup E,\Sgot_*^{n-1})$ solves the lemma. Indeed, observe that $\tau_1\tau_2=\sum_{j=3}^n\psi_j^2$ (see \eqref{eq:Delta}, \eqref{eq:FDelta}, and \eqref{eq:H2}), $[\tau_j|_K]=[f_j]$ for $j=1,2$ (see (a3), (b3), and \eqref{eq:Delta}, and recall that $h_j$ has no zeros on $K$), $[\psi_j|_K]=[f_j]$ for $j=3,\ldots,n$ (see (c2) and \eqref{eq:Delta}, and recall that $H$ vanishes nowhere on $K$), and $\tau_1$ and $\tau_2$ have no common zeroes in $\Sigma\setminus (K\cup E)$ (see \eqref{eq:common}, \eqref{eq:Delta}, \eqref{eq:soporte0}, and \eqref{eq:soporte4}); this shows that $\Psi$ has no zeros on $\Sigma\setminus E$ and takes its values in $\Sgot_*^{n-1}$ (take into account that $f$ vanishes nowhere on $K$). On the other hand, $\Psi$ is close to $\Phi$ uniformly on $K$ since $\tilde f_j$ is close to $f_j$, $j=1,\ldots,n$, and $h_1$, $h_2$, and $H$ are close to $1$ uniformly on $K$. Finally, $\Psi-\Phi$ is holomorphic on $K$, hence it lies in $\Omega(K)^n$, and it vanishes to the required order at every point in $E'\cup\Lambda$ by conditions (a2), (b2), (c1), and the fact that $h_1$, $h_2$, and $H$ are tangent to $1$ to a high order in $E'\cup\Lambda$, whenever all these interpolations are of sufficiently high order.
\end{proof}

\section{Proof of Theorem \ref{th:intro}}\label{sec:proof}



Let $\Sigma$, $\Igot$, $X$, $E$, $K$, $E'$, and $\Lambda$ be as in the statement of Theorem \ref{th:intro}.

We begin with the following reduction. The proof is essentially that of \cite[Theorems 4.4 and 4.5]{AlarconForstnericLopez2020Memoirs} but allowing poles of the differential and adding interpolation. This upgrade is well understood in the orientable framework (see \cite[Theorem 3.1]{AlarconLopez2022APDE}), and the proof is adapted to the non-orientable case in a straightforward way.
%
%
\begin{claim}\label{cl:sdr}
We can assume that $K$ is a strong deformation retract of $\Sigma\setminus E$ and $X$ is full (see Definition \ref{def:full-CMI}).
\end{claim}
\begin{proof}[Sketch of proof]
Let $\theta$ be a nowhere vanishing $\Igot$-invariant holomorphic $1$-form on $\Sigma\setminus(E\cup E')$ that is meromorphic on $\Sigma\setminus E$ (that is, $\theta\in \Omega_{\infty,\Igot}(\Sigma\setminus E|E')$) such that $f:=\partial X/\theta$ is holomorphic and nowhere vanishing on $K$, hence $f\in\Ocal_\Igot(K,\Agot_*^{n-1
})$. We can for instance choose $\theta$ of the form $\theta=g(\overline{g\circ\Igot})\theta_0$, where $\theta_0\in\Omega_\Igot(\Sigma\setminus E)$ is nowhere vanishing (see \cite[Theorem 6.4]{AlarconLopez2015GT}), $g\in\Ocal_\infty(\Sigma\setminus E|E_0')$ is nowhere vanishing and $\partial X/g$ is holomorphic and nowhere vanishing on $E_0'$, where $E'=E_0'\cup\Igot(E_0')$ with $E_0'\cap\Igot(E_0')=\varnothing$. Such a $g$ is provided by the classical Mittag-Leffler theorem. Note that $\theta$ need not extend meromorphically to the points in $E$.

Combining the arguments in the proofs of \cite[Theorems 4.4 and 4.5]{AlarconForstnericLopez2020Memoirs} and \cite[Theorem 3.1]{AlarconLopez2022APDE}, we can approximate $f$ uniformly on $K$ by a map $\tilde f\in\Ocal_\Igot(\Sigma\setminus E,\Agot_*^{n-1})$ such that $\tilde f-f$ vanishes to a high order on $E'\cup \Lambda$ and $\tilde f\theta$ integrates to an $\Igot$-invariant full conformal minimal immersion $\tilde X:\Sigma\setminus(E\cup E')\to\r^n$ such that $\tilde X-X$ vanishes to a high order on $E'\cup\Lambda$ and $\Flux_{\tilde X}(\gamma)=\Fgot(\gamma)$ for every closed curve $\gamma\subset \Sigma\setminus(E\cup E')$. Note that $\tilde X$ need not be complete or of finite total curvature. Replacing $(X,K)$ with $(\tilde X,\tilde K)$ for any $\Igot$-invariant smoothly bounded compact domain $\tilde K\subset \Sigma\setminus E$ that contains $K$ and is a strong deformation retract of $\Sigma\setminus E$, proves the claim.
\end{proof}

In the sequel, we assume that $K$ is a strong deformation retract of $\Sigma\setminus E$ and $X$ is full. Also let $\epsilon$ and $r$ be as in the statement. Let us find an immersion $Y$ satisfying the conclusion of the theorem. Choose $\theta\in \Omega_{\infty,\Igot}(\Sigma|E)$ vanishing nowhere on $K$, given by Claim \ref{cl:theta}, call
\begin{equation}\label{eq:Phi}
\Phi=\Xi\circ \partial X \in \Omega_{\infty,\Igot}(K|E',\Sgot_*^{n-1})
\end{equation}
 (see \eqref{eq:Thetabiho}) and set 
\begin{equation}\label{eq:u}
u=(u_1,\ldots,u_n)=\Phi/\theta\in  \Ocal_{\infty,\Igot}(K|E',\Sgot_*^{n-1}).
\end{equation}

\noindent{\bf Step 1: The homology basis and the period map.}
Since the $2$-sheeted covering $\Sigma\to \Sigma/\Igot$ is unique up to isomorphisms, there are smoothly bounded compact domains  $\Sigma^+,\Sigma^-=\Igot(\Sigma^+)$ in $\Sigma$ with $\Sigma=\Sigma^+\cup \Sigma^-$ and $\mathring \Sigma^+\cap \mathring \Sigma^-=\varnothing$. Furthermore, there is a such decomposition with $ \Sigma^+\cap  \Sigma^-=b\Sigma^+=  b\Sigma^-\subset \mathring K\setminus (\Lambda\cup E')$ being either a smooth Jordan curve (when the genus of $\Sigma$ is even)  or a disjoint union of two smooth Jordan curves (when the genus of $\Sigma$ is odd).  Moreover, $\Sigma\setminus K$ is a disjoint union of smoothly bounded closed discs, each containing a single point in $E$.
 See \cite[Section 2.3]{AlarconForstnericLopez2020Memoirs} for more details. Write $K^\pm=K\cap \Sigma^\pm$, $E^\pm=E\cap \Sigma^\pm$, $(E')^\pm=E'\cap \Sigma^\pm$, and $\Lambda^\pm=\Lambda\cap \Sigma^\pm$.

Up to enlarging the finite set $\Lambda$, we may assume that 
 \begin{equation}\label{eq:u1un}
   (u_1 u_2)(p)\neq 0 \quad \text{for all } \; p\in K\setminus (\Lambda \cup E').
 \end{equation}
According to \cite[Section 2.3]{AlarconForstnericLopez2020Memoirs}, there is a collection of smoothly embedded  oriented closed curves  $C_0\subset \mathring K\setminus (\Lambda\cup E')$ and $C_1^+,\ldots, C_l^+$ in $\mathring K^+\setminus (\Lambda^+\cup (E')^+)$ satisfying the following conditions.
\begin{itemize}
\item $\Igot(C_0)=C_0$ as oriented curves, hence $\Igot_*(C_0)=C_0$ homologically in $\Sigma\setminus E$.
\item  $\Bcal_0=\{C_0,C_1^+,\ldots, C_l^+, C_1^-= \Igot_*(C_1^+),\ldots,C_l^-= \Igot_*(C_l^+)\}$ is a homology basis in $\Sigma\setminus E$.
\item $b\Sigma^+=C_0$   if  the genus of $\Sigma$ is even while $\mathring K^\pm\cap C_0\neq \varnothing$  if it is odd.
\item $C=C_0\cup \bigcup_{j=1}^l (C_j^+\cup  C_j^-)$ is Runge in $\Sigma\setminus E$ (and disjoint from $\Lambda \cup E'$).
\item Any two different curves in $\Bcal_0$ intersect in at most at a single point.
\end{itemize}
(See \cite[Figures 2.1 and 2.2]{AlarconForstnericLopez2020Memoirs}.)
Write $\Lambda^+\cup (E')^+=\{q_1,\ldots,q_k\}$. 
Choosing the curves in $\Bcal_0$  in the right way, we can take a point 
\begin{equation}\label{eq:p0+}
p_0^+\in \mathring K^+\setminus (C\cup \Lambda\cup E')
\end{equation}
 and a collection of smoothly embedded oriented  arcs $\gamma_j^+\subset \mathring K^+\setminus C$, $j=1\ldots,k$, satisfying the following conditions:
\begin{itemize}
\item $\gamma_j^+$ has $p_0^+$ as initial point and $q_j$ as final point for every $j=1\ldots,k$.
\item $\gamma_j^+\cap \gamma_i^+=\{p_0^+\}$  for every  $i,j\in \{1\ldots,k\}$, $i\neq j$.
 \item Setting $\gamma_j^-=\Igot(\gamma_j^+)$ for all $j\in \{1\ldots,k\}$ and  $C'=\bigcup_{j=1}^k (\gamma_j^+\cup\gamma_j^-)$, we have that 
 \begin{equation}\label{eq:CC'}
 \text{$C'\cup C$ is Runge in $\Sigma\setminus E$.}
 \end{equation}
\end{itemize}
Set $\Bcal =\{C_0,C_1^+,\ldots, C_l^+, \gamma_1^+,\ldots, \gamma_k^+\}$ and write $\Bcal=\{\Gamma_0=C_0,\Gamma_1,\ldots,\Gamma_{\ell}\}$, where $\ell=l+k$. Call 
\begin{equation}\label{eq:Gamma}
\Gamma=\bigcup_{j=0}^\ell \Gamma_j.
\end{equation}
 \begin{definition}\label{de:Cu}
Denote by $\Ccal^0(\Gamma,u)$ the set of continuous functions $H=(H_1,\ldots,H_n)\in \Ccal^0(\Gamma\setminus (E')^+,\Sgot_*^{n-1})$ such that $H-u\in \Ccal^0(\Gamma)$,   $H_2|_{\Gamma_0}=-\overline H_1\circ \Igot|_{\Gamma_0}$, and $H_i|_{\Gamma_0}=\overline H_i\circ\Igot|_{\Gamma_0}$ for $3\leq i\leq n$.
\end{definition}
Define 
\begin{equation}\label{eq:Pcal}
\Pcal=(\Pcal_0,\Pcal_1,\ldots,\Pcal_\ell): \Ccal^0(\Gamma,u)\to \r^n\times (\c^n)^\ell
\end{equation}
 given by
\begin{multline*}
\Pcal_0(H)=\left( \Re \int_{\Gamma_0} (H_1-u_1)\theta , \Im \int_{\Gamma_0}(H_1-u_1) \theta,\right.\\ \int_{\Gamma_0}(H_3-u_3)\theta,\ldots, \left.\int_{\Gamma_0}(H_n-u_n)\theta\right)\in \r^n,
\end{multline*}
\[
\Pcal_j(H)=\left( \int_{\Gamma_j} (H_i-u_i)\theta \right)_{i=1,\ldots,n}\in \c^n, \quad 1\leq j\leq \ell.
\]
Note that $\Pcal_0$ takes its values in $\R^n$ by the $\Igot$-invariances of $\Gamma_0$ and $(H_j,u_j)$ for $j\ge 3$.
The main task to complete the proof is to find a map $H\in \Ocal_\infty(\Sigma|E\cup E'\cup \supp([\theta]))^n$ satisfying the following main conditions:
\begin{enumerate}[{\rm ({P}1)}]
\item $H \theta \in \Omega_{\infty,\Igot}(\Sigma|E\cup E',\Sgot_*^{n-1})$ and it has an effective pole at every point in $E$.
\item $H \theta-\Phi\in \Omega(K)^n$  is close to $0$ on $K$ and vanishes to a given order in $\Lambda\cup E'$.
\item  $H\in \Ccal^0(\Gamma,u)$ and  $\Pcal(H)=0$.
\end{enumerate}
This will be done in several stages (Steps 2 to 5 in what follows).

\noindent{\bf Step 2:  Construction of the local spray.} We shall first construct a $\Pcal$-dominating spray $\r^N\times (\Gamma\setminus (E')^+)\to   \c^n$  of maps in $\Ccal^0(\Gamma,u)$ with the core $u$.
 To begin with, choose $\alpha\in \c$ with 
 \begin{equation}\label{eq:alpha}
 \Re(\alpha)\Im(\alpha)\neq 0.
 \end{equation}
  Set $N= 2n(\ell+1)$, and assume that we are given a family of continuous functions $f=(f_{i,j}):\Gamma \to \c^{N}$, $1\leq i\leq 2n$, $0\leq j\leq \ell$, such that $f|_{\Gamma_0}=\bar f \circ\Igot|_{\Gamma_0}$  and 
  \begin{equation}\label{eq:suppf}
  \supp(f)\cap (\Lambda^+\cup (E')^+\cup\{p_0^+\})=\varnothing.
  \end{equation}  
   Define the spray of maps 
  \begin{equation}\label{eq:Hf}
  H_{f,u}:\r^N\times  (\Gamma\setminus (E')^+)\to \c^n
  \end{equation}
   by 
 \begin{multline*}
 H_{f,u}(\xi,\cdot)=\big((1+\alpha \xi\cdot f)^2u_1,(1+\bar \alpha \xi\cdot f)^2u_2, \\(1+\alpha \xi\cdot f)(1+\bar \alpha \xi\cdot f) u_3,
 \ldots,(1+\alpha \xi\cdot f)(1+\bar \alpha \xi\cdot f) u_n\big),
 \end{multline*}
 for $\xi=(\xi_{i,j})\in \r^N$, $1\leq i\leq 2n$, $0\leq j\leq \ell$, where $ \xi\cdot f=\sum_{i=1}^{2 n}\sum_{j=0}^\ell \xi_{i,j} f_{i,j}:\Gamma \to \c$. 
 In view of \eqref{eq:u1un},  \eqref{eq:alpha}, \eqref{eq:u} and \eqref{eq:suppf}, we have that $H_{f,u}(\xi,\cdot)\in  \Ccal^0(\Gamma,u)$ for every $\xi\in \r^N$; see Definition \ref{de:Cu}. Hence, we can consider the period map
 \[
 \Qcal_{f,u}=(\Qcal_0,\Qcal_1,\ldots,\Qcal_\ell):\r^N\to \r^n\times (\c^n)^\ell, \quad \Qcal_{f,u}(\xi)= \Pcal\big(H_{f,u}(\xi,\cdot)\big);
 \]
 see \eqref{eq:Pcal};  a vector valued quadratic polynomial.
  Suppose in addition that 
  \begin{equation}\label{eq:supf2}
\text{$\supp(f_{i,j})\subset \Gamma_j\setminus \bigcup_{k\neq j} \Gamma_k$ \quad  for every $i$ and $j$,}
\end{equation}
hence
\[
  \frac{\partial \Qcal_k}{\partial \xi_{i,j}}\Big|_{\xi=0}=0,\quad k, j\in \{0,\ldots,\ell\},\; k\neq j.
  \]
Taking into account that $u_2=-\bar u_1\circ\Igot$ (see \eqref{eq:u}) and  $f|_{\Gamma_0}=\bar f \circ\Igot|_{\Gamma_0}$, it follows for each $i=1,\ldots,2n$  that
   \begin{multline}\label{eq:JQ0}
   \frac{\partial \Qcal_0}{\partial \xi_{i,0}}\Big|_{\xi=0}=2\left(  \Re\int_{\Gamma_0} \alpha f_{i,0} u_1\theta,  \Im\int_{\Gamma_0} \alpha f_{i,0} u_1\theta,\right.\\
   \left.
  \Re(\alpha) \int_{\Gamma_0}   f_{i,0} u_3\theta,\ldots, \Re(\alpha)\int_{\Gamma_0}   f_{i,0} u_n\theta \right)\in \r^n.
  \end{multline}
Likewise, for $1\leq j\leq \ell$ we have that
  \begin{multline}\label{eq:JQj}
   \frac{\partial \Qcal_j}{\partial \xi_{i,j}}\Big|_{\xi=0}=2\left( \alpha \int_{\Gamma_j}  f_{i,j} u_1\theta, \bar\alpha\int_{\Gamma_j}   f_{i,j} u_2\theta,\right.\\
   \left.
    \Re(\alpha)\int_{\Gamma_j}  f_{i,j} u_3\theta,\ldots,  \Re(\alpha)\int_{\Gamma_j} f_{i,j} u_n\theta \right)\in \c^n.
  \end{multline}
Writing $\Qcal_j=(\Re \Qcal_j,\Im \Qcal_j):\r^N\to \r^2$ for $1\leq j\leq \ell$, it turns out that the Jacobian matrix of $\Qcal_{f,u}:\r^N\to \r^{2n\ell+ n}=\r^{N-n}$ at $\xi=0$ has the following block structure
\[
{\rm Jac}_0(\Qcal_{f,u})=\left(\begin{array}{c|c|c|c}
(W_0)_{2n\times n} & 0_{2n\times 2n} & \cdots  & 0_{2n\times 2n} \\
\hline
0_{2n\times n} & (W_1)_{2n\times 2n}  & \cdots  & 0_{2n\times 2n} \\
\hline 
\vdots & \vdots & \ddots  & \vdots\\
\hline 
0_{2n\times n} & 0_{2n\times 2n} &\cdots & (W_\ell)_{2n\times 2n} 
\end{array}
 \right)_{2n(\ell+1)\times (n+2n \ell)},
\]
 where
 \[
 W_0= \left(\frac{\partial \Qcal_0}{\partial \xi_{i,0}}\Big|_{\xi=0}\right)_{1\leq i\leq 2n}
 \]
 and
 \[
 W_j= \left(\left(\frac{\partial \Re \Qcal_j}{\partial \xi_{i,j}}\Big|_{\xi=0}\right)_{1\leq i\leq 2n} \Big| \left(\frac{\partial \Im\Qcal_j}{\partial \xi_{i,j}}\Big|_{\xi=0}\right)_{1\leq i\leq 2n}\right),\;\; 1\leq j\leq \ell.
 \] 
 Since $u$ is full,  for every $j=0,\ldots,\ell$ there are points $\{p_{i,j}\colon i=1,\ldots,n\}$ in $\Gamma_j\setminus (\Lambda^+\cup (E')^+\cup\{p_0^+\}\cup \bigcup_{k\neq j} \Gamma_k)$ such that 
 \begin{equation}\label{eq:span}
 \span_\c \{u(p_{1,j}),\ldots,u(p_{n,j})\} =\c^n.
 \end{equation}
 Further, make sure that $\Igot(\{p_{i,0}\colon i=1,\ldots,n\})\cap\{p_{i,0}\colon i=1,\ldots,n\}=\varnothing$. Choose  functions $f_{i,j}$ as above such that the following holds for $i=1\ldots,n$ and $j=0,\ldots,\ell$:
 \begin{equation}\label{eq:JQ01}
\text{$\int_{\Gamma_0} f_{2i-1,0} u \theta \approx u(p_{i,0})+u(\Igot(p_{i,0}))$ and $\int_{\Gamma_0} f_{2i,0} u \theta \approx -\imath u(p_{i,0})+\imath u(\Igot(p_{i,0}))$,}
\end{equation}
and 
 \begin{equation}\label{eq:JQj1}
 \text{$\int_{\Gamma_j} f_{2i-1,j} u \theta \approx u(p_{i,j})$ and $\int_{\Gamma_j} f_{2i,j} u \theta \approx -\imath u(p_{i,j})$.}
 \end{equation}
 For instance, we can take each $f_{i,0}$ as a suitable function satisfying  $f_{i,0}|_{\Gamma_0}=\overline f_{i,0}\circ \Igot|_{\Gamma_0}$  with the support in a small neighborhood of $\{p_{i,0},\Igot(p_{i,0})\}$, and each $f_{i,j}$,  $1\leq j\leq \ell$, with the support in a small neighborhood of $p_{i,j}$. 
  
 We claim that ${\rm Jac}_0(\Qcal_{f,u})$ has maximal rank equal to $n+2n \ell$. Indeed, let us first see that $W_0$ has rank equal to $n$. Since $u\in  \Ocal_{\infty,\Igot}(K|E',\Sgot_*^{n-1})$, for $1\leq i\leq n$  we infer from \eqref{eq:JQ0} and \eqref{eq:JQ01} that 
%
 \[
  \left(\frac{\partial \Qcal_0}{\partial \xi_{2i-1,0}}+\imath   \frac{\partial \Qcal_0}{\partial \xi_{2i,0}}\right)\Big|_{\xi=0}\approx 2 \Big(  \alpha u_1- \bar \alpha u_2  , -\imath (\alpha u_1+\bar \alpha u_2  ),
  2\Re(\alpha) u_3 ,\ldots,  2\Re(\alpha) u_n \Big) (p_{i,0}).
 \]
 Since $\Re(\alpha)\neq 0$ \eqref{eq:alpha}, equation \eqref{eq:span} ensures that the columns of $W_0$ are linearly independent whenever the above approximations are close enough, hence $W_0$ has rank equal to $n$. It remains to check that the rank of $W_j$ equals $2n$ for every $j=1,\ldots,\ell$. For each such $j$, we have in view of \eqref{eq:JQj} and \eqref{eq:JQj1} that
 \[
   \frac{\partial \Qcal_j}{\partial \xi_{2i-1,j}}\Big|_{\xi=0}\approx \imath  \frac{\partial \Qcal_j}{\partial \xi_{2i,j}}\Big|_{\xi=0}\approx 2\left( \alpha   u_1, \bar\alpha u_2,
    \Re(\alpha) u_3 ,\ldots,  \Re(\alpha) u_n  \right)(p_{i,j})\in \c^n.
  \]
  Therefore $W_j=2(\Re W_j^*|\Im W_j^*)$ where 
  \[
  W_j^*\approx \left(\begin{array}{ccccc}
 \alpha   u_1   (p_{1,j}) & \bar\alpha u_2   (p_{1,j}) &  \Re(\alpha) u_3   (p_{1,j}) & \cdots &  \Re(\alpha) u_n(p_{1,j})\\
  -\imath \alpha   u_1   (p_{1,j}) & -\imath \bar\alpha u_2   (p_{1,j}) &  -\imath \Re(\alpha)  u_3   (p_{1,j}) & \cdots &-\imath  \Re(\alpha) u_n(p_{1,j})\\
  \vdots & \vdots & \vdots & \vdots & \vdots \\
  \alpha   u_1   (p_{n,j}) & \bar\alpha u_2   (p_{n,j}) &  \Re(\alpha) u_3   (p_{n,j}) & \cdots &  \Re(\alpha) u_n(p_{n,j})\\
  -\imath \alpha   u_1   (p_{n,j}) & -\imath \bar\alpha u_2   (p_{n,j}) &  -\imath \Re(\alpha)  u_3   (p_{n,j}) & \cdots &-\imath  \Re(\alpha) u_n(p_{n,j})
  \end{array}\right)_{2n\times n}.
  \]
As above, since $\Re(\alpha)\neq 0$ \eqref{eq:alpha}, equation \eqref{eq:span} ensures that  the columns of $W_j$ are linearly independent whenever the above approximations are close enough, hence $W_j$ has rank equal to $2n$. This shows that ${\rm Jac}_0(\Qcal_{f,u})$ has  rank equal to $n+2n \ell$ as claimed.   In particular there is  a closed ball $\Delta\subset \r^N$ with $0\in \mathring \Delta$   such that the quadratic map
\begin{equation}\label{eq:Lambdaf}
\text{$\Qcal_{f,u}:\Delta\to \r^{n+2n \ell}$  is submersive   and  $0=\Qcal_{f,u}(0)\in\Qcal_{f,u}(\mathring \Delta) \subset \r^{n+2n \ell}$.}
\end{equation}
In other words, the spray $H_{f,u}:\r^N\times (\Gamma\setminus (E')^+)\to   \c^n$ of maps in $\Ccal^0(\Gamma,u)$ with the core $H_{f,u}(0,\cdot)=u|_{\Gamma\setminus (E')^+}$
is period dominating for the period map $\Pcal$. 

\noindent{\bf Step 3: First approximation of the  spray.}
 Next, we shall approximate the spray $H_{f,u}$ by a suitable $\Pcal$-dominating spray $\r^N\times (K\setminus E')\to  \c^n$ of maps in $\Ccal^0(\Gamma,u)\cap \Ocal_{\infty,\Igot}(K|E',  \Sgot_*^{n-1})$ with the core $u$. 
Call 
\begin{equation}\label{eq:Gamma^*}
\Gamma^*=\bigcup_{j=1}^\ell \Igot(\Gamma_j)=\big(\bigcup_{j=1}^l C_j^-\big)\cup \big(\bigcup_{j=1}^k \gamma_j^-\big)\subset \Sigma^-;
\end{equation}
 cf.\ \eqref{eq:Gamma} and note that $\Gamma\cap \Gamma^*=\varnothing$. Further, by \eqref{eq:CC'} we have that $\Gamma\cup \Gamma^*=C\cup C'$ is Runge in $\Sigma\setminus E$. Consider the continuous map $f^*:\Gamma\to \c^N$ given by $f^*=f$ on $\Gamma\setminus \Gamma_0$ and $f^*=f/2$ on $\Gamma_0$; see \eqref{eq:supf2}. Recall that  $\bar f\circ \Igot =f$ on $\Gamma_0$, hence    $f=f^*+\bar f^*\circ \Igot$ on $\Gamma_0$.  By the algebraic Runge theorem in     \cite[Theorem 10]{Royden1967JAM} we can approximate the map $f^*:\Gamma \to \c^N$ uniformly on $\Gamma$ by a meromorphic function $g=(g_{i,j})\in\Ocal_\infty(\Sigma|E)^N$ ($1\leq i\leq 2n$, $0\leq j\leq \ell$) that is uniformly close to $0$ in $\Gamma^*$ and vanishes to any given order on $\Lambda\cup E'$; take into account \eqref{eq:suppf} for the latter.
     Set 
   \[
   \hat g=g+\bar g\circ \Igot=(\hat g_{i,j}):\Sigma\setminus E \to \c^N
   \]
and note that: 
   \begin{enumerate}[{\rm ({A}1)}]
   \item     $\hat g\in \Ocal_{\infty,\Igot}(\Sigma|E)^N$.
   \item    $\hat g$ is uniformly close to $f$ in $\Gamma$.
   \item  $\hat g$ vanishes to a given order on $\Lambda\cup E'$. 
   \end{enumerate}
Define the spray of maps $H_{\hat g,u}:\r^N\times  (K\setminus E')\to \c^n$ with core  $H_{\hat g,u}(0,\cdot)=u$ as the map $H_{f,u}$ \eqref{eq:Hf} with $\hat g$ in the place of $f$:
 \begin{multline}\label{eq:Hg}
  H_{\hat g,u}(\xi,\cdot)=\big((1+\alpha \xi\cdot {\hat g})^2 u_1,(1+\bar \alpha \xi\cdot {\hat g})^2 u_2, \\(1+\alpha \xi\cdot {\hat g})(1+\bar \alpha \xi\cdot {\hat g})  u_3,
 \ldots,(1+\alpha \xi\cdot {\hat g})(1+\bar \alpha \xi\cdot {\hat g})  u_n\big).
 \end{multline} 
Assuming that the zero order of $\hat g$ at every point in $E'\cup \Lambda$ is high enough, see (A3), it is clear that 
  \begin{equation}\label{eq:A3}
 \text{$H_{\hat g,u}(\xi,\cdot)-u\in \Ocal(K)^n$ and vanishes to a given order in $E'\cup \Lambda$.}
 \end{equation}
Moreover,
   \begin{equation}\label{eq:A30}
 \text{$H_{\hat g,u}(\xi,\cdot)-u$ converges uniformly  on $K$ to $0$ as $\xi$ converges to $0$},
 \end{equation}
 that is, $\lim_{\xi\to 0}\|H_{\hat g,u}(\xi,\cdot)-u\|_K=0$. 
 
 We claim that $H_{\hat g,u}(\xi,\cdot)$ vanishes nowhere. Indeed, since $\Im(\alpha)\neq 0$ \eqref{eq:alpha} we have that
 $1+\alpha \xi\cdot {\hat g}$ and $1+\bar \alpha \xi\cdot {\hat g} $ have no common zeroes, and so  $H_{\hat g,u}(\xi,\cdot)$ vanishes nowhere on $K\setminus (E'\cup \Lambda)$;  see \eqref{eq:u1un}. On the other hand, $\hat g(p)=0$ and so $H_{\hat g,u}(\xi,p)=u(p)\neq 0$ for every $p\in \Lambda$ (see \eqref{eq:u}). Therefore, $H_{\hat g,u}(\xi,\cdot):  K\setminus E'\to \c^n$   takes its values in $\Sgot_*^{n-1}$, and  by \eqref{eq:A3}
 we also have that  $H_{\hat g,u}(\xi,\cdot)\in \Ccal^0(\Gamma,u)$ for every $\xi \in \r^N$. 
If the approximation of $f$ by ${\hat g}$ on $\Gamma$ is close enough,  see (A2),  the spray 
\begin{equation}\label{eq:Hgatgu}
\text{$ H_{\hat g,u}:\r^N\times (K\setminus E')\to   \Sgot_*^{n-1}$ 
 is period dominating for the period map $\Pcal$,}
 \end{equation}
 see \eqref{eq:Pcal} and recall that $H_{f,u}$ is $\Pcal$-dominating.  Furthermore,  by \eqref{eq:Lambdaf}, we can assume that  the corresponding  vector valued quadratic polynomial   period map
  \begin{equation}\label{eq:Qg}
 \Qcal_{\hat g,u}:\r^N\to \r^n\times (\c^n)^\ell, \quad \Qcal_{\hat g,u}(\xi)=  \Pcal\big( H_{\hat g,u}(\xi,\cdot)\big),
 \end{equation}
satisfies that
 \begin{equation}\label{eq:Lambdahatg}
\text{$\Qcal_{\hat g,u}:\Delta \to \r^{n+2n \ell}$  is   submersive   and  $0=\Qcal_{\hat g,u}(0)\in\Qcal_{\hat g,u}(\mathring \Delta) \subset \r^{n+2n \ell}$.}
\end{equation}
By shrinking the ball $\Delta$ if necessary,   due to  \eqref{eq:A30}, we can also  assume that 
\begin{equation}\label{eq:H-u}
\text {$|H_{\hat g,u}(\xi,\cdot)-u|(p)<\epsilon'$ for every $p\in K$ and $\xi \in \Delta$,}
\end{equation}
 where $\epsilon'>0$ is a given number that will be specified later.

%

 \noindent{\bf Step 4: Approximation of the core.}
 Now, we shall approximate the spray $H_{\hat g,u}$ with the core $u$ by another   $\Pcal$-dominating spray $\r^N\times (K\setminus E')\to  \c^n$ of maps in $\Ccal^0(\Gamma,u)\cap \Ocal_{\infty,\Igot}(K|E',  \Sgot_*^{n-1})$ whose core lies in $\Ocal_{\infty,\Igot}(\Sigma|E'\cup E\cup \supp([\theta]),\Sgot_*^{n-1})$. 
 
 Recall that  $\Phi\in \Omega_{\infty,\Igot}(K|E',\Sgot_*^{n-1})$ is full; see  \eqref{eq:Phi}. By Lemma \ref{lem:core}, there exists a full meromorphic $1$-form  $\hat \Phi\in \Omega_{\infty,\Igot}(\Sigma|E'\cup E,\Sgot_*^{n-1})$ such that:
    \begin{enumerate}[{\rm ({B}1)}]
   \item  $\hat \Phi-\Phi\in \Omega(K)$ and is uniformly close to $0$ on $K$.
   \item $\hat \Phi$ has an effective pole at every point in $E$.
   \item $\hat \Phi-\Phi$   vanishes to a given finite order   at every point in $E'\cup \Lambda$.
   \end{enumerate}
 Observe that  $H=\hat \Phi/\theta$ meets conditions  {\rm (P1)} and  {\rm (P2)} stated at the end of Step 1, hence $H\in \Ccal^0(\Gamma,u)$, but it  need not satisfy $\Pcal(H)=0$ in {\rm (P3)}. In the remainder of the proof we shall slightly modify $\hat \Phi$ to grant this condition.
 Recall that   $\theta\in \Omega_{\infty,\Igot}(\Sigma|E)$ vanishes nowhere on $K$, set 
 \begin{equation}\label{eq:hatu}
 \hat u=(\hat u_1, \hat u_2,\ldots,\hat u_n)=(\hat \Phi/\theta)\vert_{K\setminus E'}\in \Ocal_{\infty,\Igot}(K|E',\Sgot_*^{n-1}),
 \end{equation}
 and observe that $\hat u\in \Ocal_{\infty,\Igot}(\Sigma|E'\cup E\cup \supp([\theta]),\Sgot_*^{n-1})$.
     Recall also  that $\hat \Phi=(\hat \Phi_1,\hat \Phi_2,\ldots,\hat \Phi_n)$ is full,  and consider the finite set 
 \begin{equation}\label{eq:T}
 T=T^+\cup T^-=\{p\in \Sigma\setminus (E\cup E')\colon (\hat \Phi_1 \hat \Phi_2 )(p)=0\};
 \end{equation}
 here $T^\pm=T\cap \Sigma^\pm$  and $\Igot(T)=T$.   By Hurwitz theorem and \eqref{eq:u1un}, and assuming that the vanishing order in (B3)  is high enough, we can guarantee that
    \begin{equation}\label{eq:TT0}
T\cap K=\{p\in K\setminus E'\colon (u_1 u_2)(p)=0\}\subset \Lambda.
\end{equation}
 Define the spray of maps $H_{\hat g,\hat u}:\r^N\times  (K\setminus E')\to \c^n$ as the map $H_{\hat g,u}$ \eqref{eq:Hg} with $\hat u$ in the place of $u$:
 \begin{multline}\label{eq:hatHg}
 H_{\hat g,\hat u}(\xi,\cdot)=\big((1+\alpha \xi\cdot {\hat g})^2\hat u_1,(1+\bar \alpha \xi\cdot {\hat g})^2 \hat u_2, \\(1+\alpha \xi\cdot {\hat g})(1+\bar \alpha \xi\cdot {\hat g}) \hat u_3,
 \ldots,(1+\alpha \xi\cdot {\hat g})(1+\bar \alpha \xi\cdot {\hat g}) \hat u_n\big),
 \end{multline} 
 and, reasoning as in the previous step, observe that:
     \begin{enumerate}[{\rm ({C}1)}]
   \item $H_{\hat g,\hat u}(\xi,\cdot)-H_{\hat g,u}(\xi,\cdot)\in \Ocal(K)^n$ for all $\xi\in \r^N$; use (B3).
   \item $H_{\hat g,\hat u}-H_{\hat g,u}:\r^N\times  K\to \c^n$ converges   uniformly on compact subsets of $\r^N\times K$  to  $0$ as $\hat u-u$ converges uniformly on $K$  to $0$.
   \item $ H_{\hat g,\hat u}(\xi,\cdot):K\setminus E'\to \c^n$ vanishes nowhere  for every $\xi\in \r^N$; recall that $\Im(\alpha)\neq 0$ \eqref{eq:alpha} and take into account (A3),  (B3), \eqref{eq:TT0}, and \eqref{eq:u}.
   \item $H_{\hat g,\hat u}(\xi,\cdot)\in \Ccal^0(\Gamma, u)$ for all $\xi\in \r^N$; see  Definition \ref{de:Cu} and (C1), and recall that $H_{\hat g,u}(\xi,\cdot)\in \Ccal^0(\Gamma,u)$. 
      \end{enumerate}
By (C3) and (C4)  we can consider the  well defined period map
  \begin{equation}\label{eq:hatQ}
 \Qcal_{\hat g,\hat u}:\r^N\to \r^n\times (\c^n)^\ell, \quad \Qcal_{\hat g,\hat u}(\xi)=  \Pcal\big(H_{\hat g,\hat u}(\xi,\cdot)\big).
 \end{equation}
We have that:
  \begin{enumerate}[{\rm ({E}1)}]
  \item  $ \Qcal_{\hat g,\hat u}$ converges uniformly on compact subsets of $\r^N$ to $ \Qcal_{\hat g, u}$ as $\hat u-u$ converges uniformly on $K$ to $0$; see (C2),   \eqref{eq:hatQ}, and  \eqref{eq:Qg}.
\item $\Qcal_{\hat g,\hat u}:\Delta \to \r^{n+2n \ell}$  is  submersive  and  $0\in\Qcal_{\hat g,\hat u}(\mathring \Delta) \subset \r^{n+2n \ell}$; see  \eqref{eq:Lambdahatg} and take into account (E1).
    \item  $|H_{\hat g,\hat u}(\xi,\cdot)-u|(p)<\epsilon'$ for every $p\in K$ and $\xi \in \Delta$; use (C2)  and \eqref{eq:H-u}.
    \end{enumerate}

In particular, there is a point  $\xi^*\in \mathring \Delta$ such that  $\Qcal_{\hat g,\hat u}(\xi^*)=0$. We have that  the meromorphic extension $H$ of $H_{\hat g,\hat u}(\xi^*,\cdot)$ to $\Sigma$ satisfies the main conditions {\rm (P2)} and {\rm (P3)}. If the $1$-form $H\theta$, which  lies in $\Omega_{\infty}(\Sigma|E\cup E')^n$, had no zeros, then $H\theta\in \Omega_{\infty,\Igot}(\Sigma|E\cup E',\Sgot_*^{n-1})$  and $H$ would also satisfy condition {\rm (P1)}. We emphasize that this need not be the case, so we have to keep working.  
  
\noindent{\bf Step 5: Second approximation of the  spray.}
Finally, we shall approximate the spray $H_{\hat g,\hat u}$  by  a   $\Pcal$-dominating spray $\r^N\times (\Sigma\setminus (E\cup E'))\to  \c^n$ of maps in $\Ccal^0(\Gamma,u)\cap \Ocal_{\infty,\Igot}(\Sigma|E\cup E'\cup \supp([\theta]),  \Sgot_*^{n-1})$ with the same core $\hat u$ of   $H_{\hat g,\hat u}$, enabling us to settle the aforementioned  difficulty.

Consider the spray of $1$-forms $\Phi_{\hat g,\hat u}:\r^N \to \Omega_{\infty}(\Sigma|E\cup E')^n$ given by
    \begin{multline*}
 \r^N\ni\xi\longmapsto \Phi_{\hat g,\hat u}(\xi)=\big((1+\alpha \xi\cdot {\hat g})^2\hat \Phi_1,(1+\bar \alpha \xi\cdot {\hat g})^2 \hat \Phi_2, \\(1+\alpha \xi\cdot {\hat g})(1+\bar \alpha \xi\cdot {\hat g}) \hat \Phi_3,
 \ldots,(1+\alpha \xi\cdot {\hat g})(1+\bar \alpha \xi\cdot {\hat g}) \hat \Phi_n\big).
 \end{multline*} 
 Observe that $\Phi_{\hat g,\hat u}(\xi)\vert_{K\setminus E'}=H_{\hat g,\hat u}(\xi,\cdot)\theta$ for every $\xi\in \r^N$ (see \eqref{eq:hatHg} and \eqref{eq:hatu}), and so
 \begin{equation}\label{eq:TK}
 \text{$\Phi_{\hat g,\hat u}(\xi)$ vanishes nowhere on $K\setminus E'$;}
 \end{equation} 
use (C3) and the fact that $\theta$ has no zeroes on $K$. However, as we emphasized above, it 
 might have zeroes in $\Sigma\setminus (K\cup E)$, thereby failing to take its values in $\Sgot_*^{n-1}$. We shall now deal with this issue.   
 
  By \cite[Proposition 4.1]{AlarconLopez2022APDE} we can approximate $\hat g|_K\in\Ocal_{\infty,\Igot}(K)^N$ (see (A1)) uniformly on the Runge compact set $K\subset \Sigma\setminus E$ by a function $h=(h_{i,j})\in \Ocal_\infty(\Sigma|E)^N$ such that $h-\hat g$ vanishes to any given order at every point of $\Lambda\cup E'$, $h$ vanishes to any given order at every point of $ T \setminus K$ (see \eqref{eq:T} and \eqref{eq:TT0}), and each $h_{i,j}$ has an effective pole at every point of $E$. Set 
\[
\hat f=(\hat f_{i,j})=\frac12(h+\bar h\circ \Igot)\in \Ocal_{\infty,\Igot}(\Sigma|E)^N
\]
and observe that:
\begin{enumerate}[{\rm ({F}1)}]
   \item $\hat f-\hat g\in \Ocal_{\infty,\Igot} (\Sigma| E)^N$ is close to $0$ uniformly on $K$.   
   \item $\hat f-\hat g$ vanishes to any given order at every point of $\Lambda\cup E'$.
   \item  $\hat f$ vanishes to any given order at every point of $T\setminus K$; see \eqref{eq:T} and recall that $\Igot(T)=T$.
   \item  For every $p\in E$, each $\hat f_{i,j}$ has an effective pole at $p$ and the map $\{1,\ldots,2n\}\times \{0,\ldots,\ell\}\to \n$ sending $(i,j)$ to the pole order of $\hat f_{i,j}$ at $p$ is injective.
\end{enumerate}
For  ensuring (F4), we choose $h$ such that each $h_{i,j}$ has an effective pole at every point in $E$ and the map $E\times \{1,\ldots,2n\}\times \{0,\ldots,\ell\}\to \n$ sending $(p,i,j)$ to the pole order of $h_{i,j}$ at $p$ is injective. This is possible by an inspection of the proof of  \cite[Proposition 4.1]{AlarconLopez2022APDE}.

Finally, consider the spray of $1$-forms $\Phi_{\hat f,\hat u}:\r^N \to \Omega_{\infty}(\Sigma|E\cup E')^n$ given by
    \begin{multline*}
 \r^N\ni\xi\longmapsto \Phi_{\hat f,\hat u}(\xi)=\big((1+\alpha \xi\cdot {\hat f})^2\hat \Phi_1,(1+\bar \alpha \xi\cdot {\hat f})^2 \hat \Phi_2, \\(1+\alpha \xi\cdot {\hat f})(1+\bar \alpha \xi\cdot {\hat f}) \hat \Phi_3,
 \ldots,(1+\alpha \xi\cdot {\hat f})(1+\bar \alpha \xi\cdot {\hat f}) \hat \Phi_n\big).
 \end{multline*} 
 Observe that (B2) and (F4) guarantee that $\Phi_{\hat f,\hat u}(\xi)$ has an effective pole at every point of $E$ for every $\xi\in \r^N$.
We claim that 
\begin{equation}\label{eq:Phihatno0}
 	\text{$\Phi_{\hat f,\hat u}(\xi)$ has no zeroes on $\Sigma\setminus (E\cup E')$},
\end{equation} 
and consequently,
\begin{equation}\label{eq:Phihatno02}
	\Phi_{\hat f,\hat u}(\xi)\in\Omega_{\infty,\Igot}(\Sigma|E\cup E',\Sgot_*^{n-1})\quad \text{for all $\xi\in \r^N$}.
\end{equation} 
Indeed, by \eqref{eq:alpha} and \eqref{eq:T} the $1$-form $\Phi_{\hat f,\hat u}(\xi)$ vanishes nowhere on $\Sigma\setminus (E\cup E'\cup T)$. On the points of $K\cap T\subset \Lambda$ (see \eqref{eq:TT0}) we have that $\hat f=\hat g$ by (F2), and so $\Phi_{\hat f,\hat u}(\xi)=\hat\Phi_{\hat g,\hat u}(\xi)$ vanishes nowhere on $K\cap T$ by \eqref{eq:TK}; while on those of $T\setminus K$ we have that $\hat f=0$ by (F3), and hence $\Phi_{\hat f,\hat u}(\xi)=\hat \Phi$ vanishes nowhere on $T\setminus K$ since $\hat\Phi\in \Omega_{\infty,\Igot}(\Sigma|E'\cup E,\Sgot_*^{n-1})$. 

We shall now look for a point $\xi_0\in\Delta$ with $\Pcal(\Phi_{\hat f,\hat u}(\xi)/\theta)=0$. Setting the spray
\[
	H_{\hat f,\hat u}:\r^N\times (K\setminus  E')\to   \Sgot_*^{n-1},\quad H_{\hat f,\hat u}(\xi,\cdot)= \Phi_{\hat f,\hat u}(\xi)/\theta,
\]
we have that:
  \begin{enumerate}[{\rm ({G}1)}]
   \item $H_{\hat f,\hat u}(\xi,\cdot)-H_{\hat g,\hat u}(\xi,\cdot)\in \Ocal(K)^n$ for all $\xi\in \r^N$; see (F2).
   \item $H_{\hat f,\hat u}-H_{\hat g, \hat u}:\r^N\times  K\to \c^n$ converges   uniformly on compact subsets of $\r^N\times K$  to  $0$ as $\hat f$ converges uniformly on $K$  to $\hat g$; see (F1).
    \item $ H_{\hat f,\hat u}(\xi,\cdot):K\setminus E'\to \c^n$ vanishes nowhere  for every $\xi\in \r^N$; see \eqref{eq:Phihatno0}.
   \item $H_{\hat f,\hat u}(\xi,\cdot)\in \Ccal^0(\Gamma, u)$ for all $\xi\in \r^N$; see  Definition \ref{de:Cu} and  use (G1) and (C4).
      \end{enumerate}
Hence we can  consider the vectorial quadratic polynomial period map  
  \begin{equation}\label{eq:Qhatf}
 \Qcal_{\hat f,\hat u}:\r^N\to \r^n\times (\c^n)^\ell, \quad \Qcal_{\hat f,\hat u}(\xi)=  \Pcal\big(H_{\hat f,\hat u}(\xi,\cdot)\big).
 \end{equation}
 If the approximation of $\hat g$ by $\hat f$ on $K$ is close enough, then (G2), \eqref{eq:Qhatf}, \eqref{eq:hatQ}, (E2), and  (E3) ensure that
  \begin{enumerate}[{\rm ({H}1)}]
  \item $\Qcal_{\hat f,\hat u}:\Delta \to \r^{n+2n \ell}$  is submersive   and  $0\in\Qcal_{\hat f,\hat u}(\mathring \Delta) \subset \r^{n+2n \ell}$.
    \item  $|H_{\hat f,\hat u}(\xi,\cdot)-u|(p)<\epsilon'$ for every $p\in K$ and $\xi \in \Delta$.
\end{enumerate}
 Consider any $\xi_0\in \Delta$ with $\Qcal_{\hat f,\hat u}(\xi_0)=0$; see  (H1). 
Since $ H_{\hat f,\hat u}(\xi_0,\cdot)\in \Ccal^0(\Gamma,u)$ and $  \Pcal\big(H_{\hat f,\hat u}(\xi_0,\cdot)\big)=0\in  \r^{n+2n \ell}$, the $1$-form $\Phi_{\hat f,\hat u}(\xi_0)-\Phi=(H_{\hat f,\hat u}(\xi_0,\cdot)-u)\theta\in \Omega(K)^n$ is exact on $K$ (see \eqref{eq:Pcal}) and it vanishes to a given order at every point of $E'\cup \Lambda$ (see (A3), (B3), and (F2)). 
In particular, the map 
\begin{equation}\label{eq:Hfinal}
	H:=H_{\hat f,\hat u}(\xi_0,\cdot)\in\Ocal_\infty(\Sigma|E\cup E'\cup \supp([\theta]))^n
\end{equation} 
satisfies conditions (P1), (P2), and (P3) at the end of Step 1.

\noindent{\bf Step 6: Completion of the proof.}  
Recall \eqref{eq:Thetabiho}, \eqref{eq:Phihatno02}, and \eqref{eq:Hfinal}, and define
   \[
\Psi=\big(  \Xi^{-1}\circ H\big) \theta\in \Omega_{\infty,\Igot}(\Sigma|E'\cup E,\Agot_*^{n-1}).
  \]
   By \eqref{eq:Phi}, (P2), and (P3), we have that $\Psi-\partial X \in \Omega_{\Igot}(K)^n$ is exact and close to $0$ on $K$, and it vanishes to a given order at every point of $E'\cup \Lambda$. 
 Since $K$ is a strong deformation retract of $\Sigma\setminus E$, we can define the $\Igot$-invariant  conformal minimal immersion 
 \[
 Y: \Sigma\setminus (E\cup E')\to \r^n, \quad Y(p)=X(p_0^+)+\Re\int_{p_0^+}^p 2\Psi,
 \]
with $\Flux_Y=\Flux_X=\Fgot$, where $p_0^+$ is the point in  \eqref{eq:p0+}. As $\Psi-\partial X=\partial Y-\partial X$  vanishes to a given order at every point of $E'\cup \Lambda$ and  $  \Pcal(H)=0\in  \r^{n+2n \ell}$ by (P3), then $Y-X$ extends  harmonically to $K$ and it vanishes  to the desired order  at every point of $E'\cup \Lambda$.  Moreover, (P2), (B2), and (F4) ensure that $\Psi=\partial Y$ has an effective pole at every point in $E\cup E'$, and hence $Y$ is complete and of finite total curvature. Finally, property (H2) implies that 
\[
|Y-X|(p)\leq \epsilon' C(\theta)<\epsilon,
\]
provided that $\epsilon'>0$ is sufficiently small, where $C(\theta)$ is the diameter of the compact Riemannian surface $(K,|\theta|^2)$; see Claim \ref{cl:theta}.  This completes the proof.


\subsection*{Acknowledgements}
The authors are partially supported by the State Research Agency (AEI) via the grant no.\ PID2023-150727NB-I00 funded by MICIU/AEI/10.13039/501100011033 and ERDF/EU, Spain.




\medskip
\noindent Antonio Alarc\'{o}n, Francisco J.\ L\'opez

\noindent Departamento de Geometr\'{\i}a y Topolog\'{\i}a e Instituto de Matem\'aticas (IMAG), Universidad de Granada, Campus de Fuentenueva s/n, E--18071 Granada, Spain.

\noindent  e-mail: {\tt alarcon@ugr.es}, {\tt fjlopez@ugr.es}

\end{document}